\documentclass[12pt,a4paper]{amsart}
\usepackage{amsfonts}
\usepackage{amsthm}
\usepackage{amsmath}
\usepackage{hyperref}
\usepackage{amscd}
\usepackage{amssymb}
\usepackage{enumitem}
\usepackage[latin2]{inputenc}
\usepackage{t1enc}
\usepackage[all]{xy}
\usepackage[mathscr]{eucal}
\usepackage{indentfirst}
\usepackage{graphicx}
\usepackage{graphics}
\usepackage{pict2e}
\usepackage{epic}
\usepackage{tikz-cd}
\DeclareMathOperator{\id}{id}
\usepackage{amsfonts}
\usepackage{mathtools}
\usepackage{tikz}
 \usetikzlibrary{matrix,arrows,decorations.pathmorphing}
\usepackage{hyperref}
\usepackage[mathscr]{eucal}
\usepackage{indentfirst}
\numberwithin{equation}{section}
\usepackage[margin=2.9cm]{geometry}
\usepackage{epstopdf}

\newcommand{\tens}{\otimes}
\newcommand{\iso}{\simeq}
\newcommand{\defeq}{\mathrel{\mathop:}=}
\newcommand {\inj}{\hookrightarrow}

\DeclareMathOperator{\Spec}{Spec}

\newcommand {\supp}[1]{{\rm Supp}(#1)}

\newcommand {\inv}{^{-1}}
\newcommand {\integ}{\mathbb Z}

\newcommand {\C}{\mathbb C}

\DeclareMathOperator{\DD}{D}
\DeclareMathOperator{\bb}{b}
\DeclareMathOperator{\coh}{coh}
\newcommand{\Dbcoh}{\DD^{\bb}_{\coh}}
\DeclareMathOperator{\perf}{perf}

\newcommand{\DperfAbs}{\DD_{\perf}}

\DeclareMathOperator{\Image}{Im}

\newcommand{\OCal}{\mathcal{O}}

\DeclareMathOperator{\can}{can}

\theoremstyle{plain}
\newtheorem{Th}{Theorem}[section]
\newtheorem{Lemma}[Th]{Lemma}

\newtheorem{Prop}[Th]{Proposition}
\newtheorem{Prop/Def}[Th]{Proposition/Definition}

 \theoremstyle{definition}

\usepackage{amscd}
\usepackage[latin2]{inputenc}
\usepackage{t1enc}

\usepackage{mathtools}
\usepackage{tikz}
 \usetikzlibrary{matrix,arrows,decorations.pathmorphing}
\usepackage{hyperref}
\usepackage[mathscr]{eucal}
\usepackage{indentfirst}

\DeclareMathAlphabet{\mathantt}{OT1}{antt}{li}{it}

\usepackage{graphicx}
\usepackage{graphics}
\usepackage{pict2e}

\usepackage{epic}

\usepackage[all]{xy}

\numberwithin{equation}{section}
\usepackage[margin=2.9cm]{geometry}
\usepackage{epstopdf}

\theoremstyle{plain}

 \theoremstyle{definition}

\DeclareMathOperator{\FMP}{FMP}

\usepackage{color}   
\usepackage{hyperref}
\hypersetup{
    colorlinks=false, 
    linktoc=all,     
    linkcolor=blue,
}

\begin{document}

\title[countability of relative Fourier-Mukai Partners]{countability of relative Fourier-Mukai Partners}

\author[R. Kurama]{Riku Kurama}

\address{Department of Mathematics, University of Michigan, Ann Arbor, MI 48109} 

\email{rkurama@umich.edu}
\maketitle
\begin{abstract}
Anel and To\"en proved that a smooth projective complex variety has only countably many smooth projective Fourier-Mukai partners up to isomorphism. This is generalized in the Stacks Project to the case where the varieties are smooth proper over an arbitrary algebraically closed field. 
This note will upgrade the proof of the latter reference to show that a smooth proper scheme over 
    a noetherian base has only countably many relative Fourier-Mukai partners up to isomorphism. 
\end{abstract}
\section{Introduction}
In \cite{Kawa}, Kawamata conjectured that a smooth projective variety over $\C$ has only finitely many smooth projective Fourier-Mukai partners up to isomorphism and verified the conjecture for complex surfaces. The conjecture was also confirmed in other cases like abelian varieties, but Lesieutre found a 3-dimensional counterexample in \cite{Lesi}, so the conjecture is now known to be false in general. 
On the other hand, Anel and To\"en proved that there are only countably many Fourier-Mukai partners in \cite{Toen}. Here is a slight strengthening of this result in \cite[\href{https://stacks.math.columbia.edu/tag/0G11}{Tag 0G11}]{stacks-project}:
\begin{Th}[{\cite{Toen}},{\cite[\href{https://stacks.math.columbia.edu/tag/0G11}{Tag 0G11}]{stacks-project}}]\label{Toen}    
Let $X$ be a smooth proper $k$-scheme, where $k$ is an algebraically closed field field. Then, $X$ has at most countably many smooth proper Fourier-Mukai partners up to isomorphism. 
\end{Th}
The motivation for this note came from the author's previous work \cite{Kurama}.
Let us quickly recall the relevant results.
Two smooth proper schemes $X,Y$ over a noetherian base $S$ are called \textbf{$S$-relatively Fourier-Mukai equivalent} if there exists a kernel $E\in \DperfAbs (X\times_SY)$ which defines for any $s\in S$ a derived equivalence
\[\Phi_{E_s}(-)\defeq R{q_s}_*(L{p_s}^*(-)\tens^L E_s):\Dbcoh(X_s)\longrightarrow\Dbcoh(Y_s),\]
 where $E_s\defeq E|_{X_s\times_{\{s\}}Y_s}$, $p_s:X_s\times_{\{s\}}Y_s\longrightarrow X_s,$ and $ q_s:X_s\times_{\{s\}}Y_s\longrightarrow Y_s$
(see \cite[\S2]{Kurama} for more details).
Let $k$ be an algebraically closed field of characteristic $p>0$, and let $X$ be an ordinary abelian variety or an ordinary K3 surface over $k$ (in the latter case, we assume $p>2$). For a smooth projective morphism $X\longrightarrow S$ of noetherian schemes, let $\FMP(X/S)$ denote the set of isomorphism classes of smooth projective $S$-schemes which are $S$-relatively Fourier-Mukai equivalent to $X$. 
Then, we have:
\begin{Th}[{{\cite[Theorem 1.5]{Kurama}}}]
Let $X_{\can}$ be the canonical lift of $X$ over the ring of Witt vectors $W(k)$. Then, the restriction to the special fiber defines a bijection
\[
\FMP(X_{\can}/W(k))\longrightarrow \FMP(X/k).
\] 
In particular, $\FMP(X_{\can}/W(k))$ is finite. 
\end{Th}
The author then asked in \cite{Kurama} whether the map $
\FMP(X_{\can}/W(k))\longrightarrow \FMP(X/k)$
is bijective for other smooth projective varieties $X$ over $k$ which admit canonical lifts over $W(k).$ If this is the case, Theorem \ref{Toen} would imply that $\FMP(X_{\can}/W(k))$ is countable.
It is then natural to ask if one can prove the countability of relative Fourier-Mukai partners in some generality.
The main result of this note answers this question for noetherian base schemes:
\begin{Th}\label{mainTh}
    Let $B$ be a noetherian scheme, and let $X$ be a smooth proper $B$-scheme. Then, there are only countably many smooth proper $B$-schemes which are $B$-relative Fourier-Mukai partners of $X$ up to isomorphism.
\end{Th}
The idea of its proof is rather simple; namely we adapt the proof of Proposition \ref{Toen} to the situation where we work relative to a noetherian base. 
Our proof is very similar to the one in \cite{stacks-project}, and there are only two main steps in the generalization.
The first one is the approximation of the base scheme by
finite type $\integ$-schemes and the use of ``models'' over such schemes, which is recalled in \S \ref{secAppMod}.
The second one is the adaptation of \cite[\href{https://stacks.math.columbia.edu/tag/0G0R}{Tag 0G0R}]{stacks-project} to our situation, which is done in Lemma 
\ref{keyLemma}.
\subsection*{Acknowledgement}
The author would like to thank his advisor Alexander Perry for continual support, helpful  discussions and suggestions. 
\section{Approximation and Models}\label{secAppMod}
This section is a quick reminder about approximation and models.
Let $B$ be a noetherian base scheme. 
Let 
\[\{X^a\}_{a\in A}, \{f^b:X^{a_{b}^1}\longrightarrow X^{a_{b}^2}\}_{b\in B}
\]
be respectively a collection of finite type $B$-schemes and a collection of $B$-linear morphisms between them.
In the resulting diagram, we may also demand a fixed collection of commutativity conditions. We also record the subsets $S\subset B$ (resp.\ $P\subset B$) which are the indices of the morphisms $f^b$ that are assumed to be smooth (resp.\ proper) (we allow some other morphisms to be also smooth or proper).
In this section, we will remind the reader that we can find a ``model'' of this diagram with desired properties over a finite type $\integ$-scheme instead of the original base scheme $B$. 

\begin{itemize}
    \item[$\bullet$ Schemes:]
First, by the same argument as noetherian approximation, we write $B$ as the inverse limit of finite type $\integ$-schemes $\{B_i\}_{i\in I}$, where $I$ is a directed set, and the transition maps of the system are affine (\cite[\href{https://stacks.math.columbia.edu/tag/01ZA}{Tag 01ZA}]{stacks-project}). 
Pick one index $a\in A$. By (1) of \cite[\href{https://stacks.math.columbia.edu/tag/01ZM}{Tag 01ZM}]{stacks-project}, we can find an index $i_0\in I$ and a finite type $B_{i_0}$-scheme $X^a_{i_0}$ which is a \textbf{model of $X^a$ over $B_{i_0}$} in the sense that $X_{i_0}^a\times_{B_{i_0}}B\iso X^a$. We replace the indexing set $I$ by $\{i\in I\;|\;i\geq i_0\}$ and define $X_i^a\defeq X_{i_0}^a\times_{B_{i_0}}B_i$ for each $i,$ so that we have an inverse system $\{X^a_{i}\}_{i\in I}$ of models of $X^a$.
We repeat the same procedure for other indices $a\in A$ as well.
\item[$\bullet$ Morphisms:]
Choose an index $b\in B$. By (2) of \cite[\href{https://stacks.math.columbia.edu/tag/01ZM}{Tag 01ZM}]{stacks-project}, we find an index $i_0\in I$ and a morphism $f^b_{i_0}:X^{a^1_b}_{i_0}\longrightarrow X^{a^2_b}_{i_0}
 $ which is a model of $f^b$ in the sense that its base change along $B\longrightarrow B_{i_0}$ is $f^b$. We replace $I$ by $\{i\in I\;|\;i\geq i_0\}$ and set $f_i^b:X_i^{a^1_b}\longrightarrow X_i^{a^2_b}$ as the base change of $f^b_{i_0}$ along $B_i\longrightarrow B_{i_0}$ for each $i$.
If $b\in S$, $f^b$ is smooth, so there is some index $i_0\in I$ for which $f_{i_0}^b$ is smooth by \cite[\href{https://stacks.math.columbia.edu/tag/081D}{Tag 081D}]{stacks-project}. We replace $I$ by $\{i\in I\;|\;i\geq i_0\}$ again. Note that $\{f^b_i\}_i$ now consists of smooth morphisms regardless of $i$.
If $b\in P$, we repeat the same procedure with respect to properness using \cite[\href{https://stacks.math.columbia.edu/tag/081F}{Tag 081F}]{stacks-project}.
 Finally, we repeat the same procedures for other indices $b\in B$ as well.
 \item[$\bullet$ Commutativity:] For each commutativity condition imposed on the original diagram, we can find an index $i_0\in I$ for which the corresponding commutativity condition holds for models over $B_{i_0}$ by (3) of \cite[\href{https://stacks.math.columbia.edu/tag/01ZM}{Tag 01ZM}]{stacks-project}. The same commutativity condition holds for all indices $i\geq i_0$ as such diagrams arise by base change. We replace $I$ by $\{i\in I\;|\; i\geq i_0\}$. We run this procedure for every commutativity condition of the original diagram.
\end{itemize}
After these steps, if we choose any index $i\in I$, we get a model \[\{X^a_i\}_a,\{f_i^b:X_i^{a^1_{a_b}}\longrightarrow X_i^{a^2_{a_b}}\}_b\] of the original diagram over the finite type $\integ$-scheme $B_i$. Moreover, such a model retains the same kind of commutativity, smoothness and properness.
\section{Proof of the theorem}
Let $B$ be a noetherian scheme. 
\begin{Lemma}[see {\cite[\href{https://stacks.math.columbia.edu/tag/0G0X}{Tag 0G0X}]{stacks-project}}]\label{LemIsom}
Let $S$ be a finite type $B$-scheme, and let $X,Y$ be finite type $S$-schemes. There exists a countable family of finite type $S$-schemes $\{S_i
 \}_{i\in I}$ such that 
 \begin{enumerate}
     \item $X_{S_i}\iso Y_{S_i}$ for each $i.$
     \item given any $B$-point $s:B\longrightarrow S$ such that $X_s\iso Y_s$ (where $X_s\defeq X\times_{S,s}B$ and similarly for $Y_s$), the map $s:B\longrightarrow S$ factors through some $S_i$.
 \end{enumerate}
\end{Lemma}
\begin{proof}
We first apply the discussions of the previous section to the diagram \[X\longrightarrow S\longleftarrow Y\] over $B$, so that we find a finite type $\integ$-scheme $A$, a morphism $B\longrightarrow A$ and a model \[X_A\longrightarrow S_A\longleftarrow Y_A\] 
of the diagram above $A$. Note that all the schemes appearing in the model diagram are finite type $\integ$-schemes.
By \cite[\href{https://stacks.math.columbia.edu/tag/0G0U}{Tag 0G0U}]{stacks-project},
   the family 
   \[\{\phi_i:T_i\longrightarrow S_A, h_i:(X_A)_{T_i}\iso (Y_A)_{T_i}\}_i\]
   of all the pairs (up to isomorphism), where $T_i$ is a finite type $S_A$-schemes and $h_i$ is a $T_i$-linear isomorphism, is countable. We will show that the base-changed family 
   \[\{\psi_i:S_i\defeq
    T_i\times_{A}B\longrightarrow S\}_i\]
    has the second desired property, as the first one is clear.
    
    Choose any section $s:B\longrightarrow S$ with an isomorphism $h:X_s\iso Y_s$. Then, the argument from the last section shows that we can find a model for the section as well as the isomorphism. 
    This means that we have a finite type $\integ$-scheme $B_0$ with a morphism $B\longrightarrow B_0$, a $B_0$-model $X_0\longrightarrow S_0\longleftarrow Y_0$ for the diagram $X\longrightarrow S\longleftarrow Y$, and $B_0$-models $s_0:B_0\longrightarrow S_0, h_0:(X_0)_{s_0}\iso (Y_0)_{s_0}$ for the section $s$ and the isomorphism $h_0.$
    Since $A$ and $B_0$ are taken from the same directed set as in \cite[\href{https://stacks.math.columbia.edu/tag/07SU}{Tag 07SU}]{stacks-project}, up to enlarging $B_0$, we may assume that the morphisms $B\longrightarrow B_0, B\longrightarrow A$ factor through $ B_0\longrightarrow A$. Since $S_A$ and $S_0$ are models of $S$, up to enlarging $B_0$ again, we may assume $(S_A)_{B_0}\iso S_0$. We ensure the same compatibilities for the models of $X,Y$ and models of morphisms as well. 
    Now, consider the map $s_0': B_0\longrightarrow S_0\iso (S_A)_{B_0}\longrightarrow S_A$. Then, we have 
    \[
    (X_A)_{s_0'}\iso (X_0)_{s_0}\iso (Y_0)_{s_0}\iso (Y_A)_{s_0'},
    \] so $s_0': B_0\longrightarrow S_A$ appears in the family $\{\phi_i:T_i\longrightarrow S_A\}_i$. Let $i$ be an index such that $s_0'$ coincides with the map $\phi_i:T_i\longrightarrow S_A$ from the family.
    We claim that $s:B\longrightarrow S$ factors through $\psi_i: B_0\times_{A}B\longrightarrow S$.
    More precisely, let $f:B\longrightarrow B_0\times_{A}B$ be the dotted arrow as in the diagram defined by the universal property of fiber product:
    \[\tag{*}\label{*}
    \xymatrix{
B\ar@{-->}[dr]_f\ar@{=}[drr]\ar[ddr]\\
&B_0\times_A{B}\ar[r]\ar[d]&B\ar[d]\\
&B_0\ar[r]&A    
    }
    \]
    and we claim that $\psi_i\circ f=s.$
    We consider the following commutative diagram where the squares are cartesian:
    \[\tag{**}\label{**}
    \xymatrix{
B\ar@{-->}[dr]\ar@{.>}[drr]\ar[ddr]_q\ar@{=}[rrrd]\\
&B_0\times_A{B}\ar[r]_-{\psi_i}\ar[d]&S\ar[r]_t\ar[d]^{v\circ u}&B\ar[d]\\
&B_0\ar_{s_0'=\phi_i}[r]&S_A\ar[r]_r&A    
    }
    \]
    Here, the dotted arrow and the thinly dotted arrows are defined by universality of fiber product (see the diagram below for the meaning of the label $v\circ u$). 
    It suffices to show that the dotted arrow is $f$ and that the thinly dotted arrow is $s$.
    For this, we consider the following commutative diagram with cartesian squares:
    \[\tag{***}\label{***}
    \xymatrix{
B\ar[r]^s\ar[d]^q&S\ar[r]^t\ar[d]^u&B\ar[d]\\
B_0\ar[r]^{s_0}\ar[dr]_{\phi_i}&S_0\ar[r]\ar[d]^v&B_0\ar[d] \\
&S_A\ar[r]^r&A
    }
    \]
    Moreover, the horizontal compositions in this diagram are the identity maps. 
    To see that the dotted arrow in (\ref{**}) is $f$, we just need to show that the big cartesian square in (\ref{**}) coincides with the cartesian square in (\ref{*}). This is easy in view of the diagram (\ref{***}) as it shows that $r\circ\phi_i$ is the same as the composition of the ``transition map'' $B_0\longrightarrow A$ with $\id_{B_0}$. 
To see that the thinly dotted arrow is $s$, it suffices to show that $s$ fits in place of the thinly dotted arrow, making the diagram commute when we delete $B_0\times_AB$ and adjacent arrows from the diagram. Now, the main thing to check is $(v\circ u)\circ s=\phi_i\circ q$, but it's clear from the diagram (\ref{***}). Hence, we see the factorization $\psi_i\circ f=s$ from the diagram (\ref{**}).
\end{proof}
\begin{Lemma}[see {\cite[\href{https://stacks.math.columbia.edu/tag/0G0R}{Tag 0G0R}]{stacks-project}}]\label{keyLemma}
   Let $S$ be a finite type $B$-scheme and let $Y$ be a smooth proper $S$-scheme. Assume that we have a section $s:B\longrightarrow S$ (which is a locally closed immersion). 
    We let $Y_s\defeq Y\times_{S,s}B$ be the fiber above $s$ (which is a locally closed subscheme of $Y$), and we define the smooth proper $S$-scheme $X\defeq Y_s\times_BS$.
    Suppose $K\in\DperfAbs(X\times_SY)$ is a kernel defining a relative Fourier-Mukai equivalence over $S$, such that we have $K|_{(X\times_SY)_s}\iso\OCal_{\Delta_{Y_s}}$ on $(X\times_SY)_s=Y_s\times_BY_s$.
    Then, there is an open $U\subset S$ containing the image of $s$ such that $Y|_U\iso Y_s\times_B U$.
\end{Lemma}
\begin{proof}
We write $Z\defeq X\times_SY$.
We note that the section $s:B\longrightarrow S$ is only a locally closed immersion as $S$ may not be separated over $V$, but we know that $\Delta_{Y_s}:Y_s\longrightarrow (X\times_SY)_s=Y_s\times_BY_s$ is a closed immersion as $Y_s$ is separated over $B$. For each closed point $c\in B$, we write $Z_c$ to denote the fiber of $Z_s\longrightarrow B$ above $c$ (we will similarly use the notation $Y_c$). We observe that the composition $\{c\}\inj B\longrightarrow S$ is a closed immersion since it factors as $\{c\}\inj S_c\inj S$.
With this notation, we have $\supp{K}\cap Z_c=\supp{K|_{Z_c}}$. 

Let $z\in Z_c$ be any closed point. 

If $z\not\in \Image(\Delta_{Y_c})\subset Z_c$, using that $K|_{Z_s}\iso\OCal_{\Delta_{Y_s}}$, we see $z\not\in\supp{K}\subset Z$. We then define the open neighborhood $U(z)\defeq Z\setminus\supp{K}\subset Z$ of $z$, which has the property $K|_{U(z)}=0$.
We also set $Z(z)=\emptyset.$

If $z\in \Image(\Delta_{Y_c})\subset Z_c$, we do the following. 
Since $Y_c$ is smooth over $\{c\}$, the map $\Delta_{Y_c}$ is a regular immersion. 
Let $\overline{f_1},..,\overline{f_r}\in\OCal_{Z_c,z}$ be a regular sequence cutting out the ideal sheaf of the closed subscheme $\Delta_{Y_c}\subset Z_c$. Then, the Koszul complex of the above sequence represents the complex $K\tens_{\OCal_{Z,z}}^L\OCal_{Z_c,z}$. By \cite[\href{https://stacks.math.columbia.edu/tag/0G0N}{Tag 0G0N}]{stacks-project}, we can lift the $\overline{f_i}$'s to a regular sequence $f_i$'s in $\OCal_Z,z$ such that $\OCal_{Z,z}/(f_1,..,f_r)$ is flat over $\OCal_{S,c}$ (where $c$ means the image of $\{c\}\inj B\longrightarrow S$). By spreading out (\cite[\href{https://stacks.math.columbia.edu/tag/0G0P}{Tag 0G0P}]{stacks-project}), we can lift $f_i$'s to sections in an affine open $U(z)\subset Z$ with similar properties so that the closed subscheme $Z(z)\defeq V(f_1,..,f_r)\subset U(z)$ has the following properties:
\begin{enumerate}
    \item $Z(z)\inj U(z)$ is a regular closed immersion
    \item $\OCal_{Z(z)}\iso K|_{U(z)}$ over $U(z)$
    \item $Z(z)\longrightarrow S$ is flat
    \item $Z(z)_c=\Delta_{Y_c}\cap U(z)_s$ as closed subschemes of $U(z)_s.$
\end{enumerate}
If we choose another closed point $c'\in B$ and a closed point $z'\in Z_{c'}$, we have 
$Z(z)\cap U(z')=Z(z')\cap U(z)$ as closed subschemes of $U(z)\cap U(z')$ in view of property (2) above (note that this works even when $z'\not\in\Image(\Delta_{Y_{c'}})$ since we will then have $U(z')\cap Z(z)=\emptyset$).
This lets us glue various $Z(z)$'s to a closed subscheme $\widetilde{Z}\subset\widetilde{U}$ of $ \widetilde{U}\defeq \bigcup_{c\in B^\circ,z\in Z_c^\circ}U(z)\subset Z$, where $B^\circ$ (resp.\ $Z_c^\circ$) is the set of closed points of $B$ (resp.\ $Z_c^\circ$).
Letting $\pi:Z\longrightarrow S$ be the structure map, we claim that the section $s:B\longrightarrow S$ lands in $V\defeq S\setminus\pi(Z\setminus\widetilde{U})$. Here, $V$ is open since $\pi$ is proper. By construction, $\widetilde{U}
$ contains every closed point of the variety $Z_c$ for each closed point $c\in B$. It follows that the open $\widetilde{U}\cap Z_c\subset Z_c$ must be all of $Z_c$, i.e.\ $Z_c\subset \widetilde{U}$. 
This means that the open $s\inv(V)\subset B$ contains every closed point $c\in B$. Then, since $B\setminus s\inv(V)$ is quasi-compact, it must be empty. This shows that the section $s$ lands in $V$.
We replace $S$ by $V$ and accordingly base-change relevant scheme over $S$. This ensures in particular that $\widetilde{Z}\inj Z$ is a closed immersion, so that we have proper maps $\alpha:\widetilde{Z}\inj Z\longrightarrow X$, $\beta:\widetilde{Z}\inj Z\longrightarrow Y$.  
For each closed point $c\in B$, these maps induce isomorphisms
$\widetilde{Z}_c\iso X_s$, $\widetilde{Z}_c\iso Y_s$, so \cite[\href{https://stacks.math.columbia.edu/tag/0G0Q}{Tag 0G0Q}]{stacks-project} lets us find open neighborhoods $U^{(c)}\subset S$ of $c$ such that $\alpha,\beta$ induce isomorphisms $\alpha|_{U^{(c)}}:(\widetilde{Z})_{U^{c}}\iso X_{U^{c}},\beta|_{U^{(c)}}:(\widetilde{Z})_{U^{c}}\iso Y_{U^{c}}$.
The open $U\defeq\cup_{c\in B^\circ}U^{(c)}\subset S$ contains the image of $s$ and has the property that $X_U\iso (\widetilde{Z})_{U}\iso Y_U$, so we are done.
\end{proof}
\begin{Prop}[see {\cite[\href{https://stacks.math.columbia.edu/tag/0G0S}{Tag 0G0S}]{stacks-project}}]\label{keyProp}
    Let $S$ be a finite type $B$-scheme, and let $Y\longrightarrow S$ and $P\longrightarrow B$ be smooth proper morphisms. Set $X=P\times_BS$. Let $K\in\DperfAbs(X\times_SY)$ define an $S$-relative Fourier-Mukai equivalence between $X$ and $Y$. If $s:B\longrightarrow S$ is a section, there is an open neighborhood $U\subset S$ of the image of $s$ such that $Y|_U\iso Y_s\times_BU$.
\end{Prop}
\begin{proof}
Let $K'$ be the kernel for the inverse transform which exists by the argument of Proposition/Definition 2.5 \cite{Kurama}.  
    $K'|_{(Y\times_SX)_s}$ defines a $B$-relative Fourier-Mukai equivalence $\Dbcoh(Y_s)\iso \Dbcoh(X_s)$, and we can use it to construct a trivial family of Fourier-Mukai equivalences from $Y_s\times_BS$ to $X_s\times_BS$ over $S$. Up to composing $K$ with this trivial family, we may replace $X_s\times_BS$ by $Y_s\times_BS$. This way, we assume $P=Y_s$. We are done by Lemma \ref{keyLemma}.
\end{proof}

\begin{proof}[Proof of Theorem \ref{mainTh}]
By \cite[\href{https://stacks.math.columbia.edu/tag/0G0U}{Tag 0G0U}]{stacks-project}, the family of collections (up to isomorphism)
\[\mathfrak{F}_0=\{(X_i\longrightarrow S_i\longleftarrow Y_i,E_i\in\DperfAbs(X_i\times_{S_i}Y_i))\}_i\]
 which parameterizes all the instances of ``relative Fourier-Mukai equivalences between smooth proper schemes $X_i$ and $Y_i$ over finite type $\integ$-schemes $S_i$'' is countable.
    We let 
    \[
    \mathfrak{F}=
    \{(X_i'\longrightarrow S_i'\longleftarrow Y_i',E_i'\in\DperfAbs(X_i'\times_{S_i'}Y_i'))\}_i
    \]
    denote the family obtained by base-changing the above family along $B\longrightarrow\Spec\integ$.

    We claim that given any smooth proper $B$-relative Fourier-Mukai partners $X,Y$ and a particular kernel $K$ giving the equivalence, there is some $B$-point $s:B\longrightarrow S_i'$ such that $(X\longrightarrow B\longleftarrow Y,K)$ arises by pulling back a member of $\mathfrak{F}$ via $s$ (see also \cite[\href{https://stacks.math.columbia.edu/tag/0G0Y}{Tag 0G0Y}]{stacks-project}).
    First, as in the previous section, 
    we find a model $(X_A\longrightarrow A\longleftarrow Y_A)$ of $(X\longrightarrow B\longleftarrow Y)$ over some finite type $\integ$-scheme $A$.
    We can also find a model $K_A\in\DperfAbs(X_A\times_{A}Y_A)$ for the kernel $K$ (in a way that $K_A$ defines a relative Fourier-Mukai equivalence from $X_A$ to $Y_A$ over $A$) as well by \cite[\href{https://stacks.math.columbia.edu/tag/0G0L}{Tag 0G0L}]{stacks-project}. 
    By construction of $\mathfrak{F}_0$, $(X_A\longrightarrow A\longleftarrow Y_A, K_A)$ appears in the family $\mathfrak{F}_0$, so we write $A=S_i$ for some index $i$.
    The corresponding scheme in the base-changed family $\mathfrak{F}$ is $S_i'=A\times_\integ B$. Since the base change of $K_A=E_i$ along $B\longrightarrow A$ is $K$, the claim is justified once we show that the map $B\longrightarrow A$ factors through $S_i'$, but universality of fiber product gives us the factorization as follows:
    \[
    \xymatrix{
B\ar[dr]\ar[ddr]\ar@{=}[drr]\\
    &S_i'\ar[r]\ar[d]&B\ar[d]\\
    &A=S_i\ar[r]&\Spec\integ
    }
    \]

    Now, we fix a smooth proper $B$-scheme $P$, and we count its $B$-relative Fourier-Mukai partners. By the claim above, 
    we only need to consider the schemes of form $(Y_i')_s$ for various sections $s:B\longrightarrow S_i'$ for various indices $i$ such that $(X_i')_s\iso P.$ By the countability of the family $\mathfrak{F}$, we only need to consider a particular index $i$.
    We thereby pick one collection from our family and for notational simplicity rewrite it as $(X\longrightarrow S\longleftarrow Y,E)$.
    We apply Lemma \ref{LemIsom} to the two $S$-schemes
    $P\times_BS, X$
and find 
 a countable family $\{S_i\longrightarrow
     S\}_i$ as in the lemma. 
     Any section $s:B\longrightarrow S$ with $X_s\iso P$ factors through one of the $S_i$'s, and this family is countable, so we may replace the base $S$ by one of the $S_i$. This way, we assume that $X$ is the trivial family $X= P\times_BS$.
     
     We claim that there are only finitely many 
     $Y_s$'s (up to isomorphism) which arise by pulling back $Y$ along sections $s:B\longrightarrow S$.
     Indeed, if $s$ is one such section, Proposition \ref{keyProp} says that there is some open neighborhood $U(s)\subset S$ of the image of $s$ over which $Y$ is a trivial family. Suppose there are infinitely many non-isomorphic $Y_s$'s, which we call $\{Y_{s_i}\}_i$. Then, we have infinitely many nonempty disjoint opens $U{(s_i)}$ of $S$ (they are disjoint because otherwise some $Y_{s_i}$'s will be isomorphic to each other).
     However, the union of all these opens is an open subscheme in $S$ which is then quasicompact, as $S$ is noetherian.
     It follows that there are only finitely many indices $i$, which is a contradiction.
\end{proof}
\bibliographystyle{alpha}
\bibliography{bib}

\end{document}